\documentclass[a4paper, 11pt,reqno]{amsart}

\usepackage[T1]{fontenc}      
\usepackage[foot]{amsaddr}    

\usepackage[british]{babel}   
\usepackage{csquotes}         
\usepackage[babel]{microtype} 

\usepackage{mathtools}        
\usepackage{amssymb}          
\usepackage{amsthm}           
\usepackage{amstext}          
\usepackage{mleftright}       
\usepackage{gensymb}          
\usepackage[makeroom]{cancel} 
\usepackage{mathdots}         
\usepackage{mathrsfs}         

\usepackage{stickstootext}
\usepackage[stickstoo,varbb]{newtxmath} 
\makeatletter
\DeclareFontFamily{U}{ntxmia}{\skewchar \font =127}
 \DeclareFontShape{U}{ntxmia}{m}{it}{
                        <-> \ntxmath@scaled ntxmia
                      }{}    
                      \DeclareFontShape{U}{ntxmia}{b}{it}{
                        <-> \ntxmath@scaled ntxbmia
                      }{}
\makeatother

\usepackage{graphicx}         
\usepackage{psfrag}           
\usepackage{caption}          
\usepackage{tikz}             
\usetikzlibrary{backgrounds}

\usepackage{booktabs}         

\usepackage{enumitem}         

\usepackage[margin=1in]{geometry} 

\usepackage[textsize=footnotesize]{todonotes}

\usepackage[numbers,sort]{natbib}

\makeatletter
\def\NAT@spacechar{~}
\makeatother

\usepackage{hyperref}              
\usepackage{cleveref}  
\hypersetup{colorlinks=true,
   citecolor=blue,
   filecolor=blue,
   linkcolor=blue,
   urlcolor=blue
}
\usepackage{url}

\crefname{figure}{figure}{figures}
\Crefname{figure}{Figure}{Figures}

\usepackage{algorithm}
\usepackage{algpseudocode}

\allowdisplaybreaks

\newtheorem{definition}{Definition}[section]
\newtheorem{claim}{Claim}

\newtheorem{proposition}[definition]{Proposition}
\newtheorem{theorem}[definition]{Theorem}
\newtheorem{corollary}[definition]{Corollary}
\newtheorem{lemma}[definition]{Lemma}

\newtheorem{question}[definition]{Question}

\newtheorem{remark}[definition]{Remark}

\newenvironment{claimproof}{%
\let\origqed=\qedsymbol%
\renewcommand{\qedsymbol}{$\blacktriangleleft$}%
\begin{proof}}{\end{proof}\let\qedsymbol=\origqed}

\numberwithin{equation}{section}

\renewcommand{\binom}[2]{\ensuremath{\mleft(\kern-.1em\genfrac{}{}{0pt}{}{#1}{#2}\kern-.1em\mright)}}    
\newcommand{\inbinom}[2]{\ensuremath{\bigl(\kern-.1em\genfrac{}{}{0pt}{}{#1}{#2}\kern-.1em\bigr)}} 

\newcommand*\nume{\ensuremath{\mathrm{e}}}

\makeatletter
\def\moverlay{\mathpalette\mov@rlay}
\def\mov@rlay#1#2{\leavevmode\vtop{%
  \baselineskip\z@skip \lineskiplimit-\maxdimen
  \ialign{\hfil$\m@th#1##$\hfil\cr#2\crcr}}}
\newcommand{\charfusion}[3][\mathord]{
    #1{\ifx#1\mathop\vphantom{#2}\fi
        \mathpalette\mov@rlay{#2\cr#3}
      }
    \ifx#1\mathop\expandafter\displaylimits\fi}
\makeatother
\newcommand{\cupdot}{\charfusion[\mathbin]{\cup}{\cdot}}

\newcommand{\COMMENT}[1]{}
\newcommand{\COMNEW}[1]{}
\renewcommand{\COMNEW}[1]{\footnote{\textcolor{red!70!black}{#1}}} 

\title[Powers of Hamilton cycles in dense graphs perturbed by a random geometric graph]{Powers of Hamilton cycles in dense graphs\linebreak{} perturbed by a random geometric graph}

\author[A.~Espuny D\'iaz]{Alberto Espuny D\'iaz}
\email{alberto.espuny-diaz@tu-ilmenau.de}
\address[Espuny D\'iaz]{Institut f\"ur Mathematik, Technische Universit\"at Ilmenau, 98684 Ilmenau, Germany.}
\author[J.~Hyde]{Joseph Hyde}
\email{josephhyde@uvic.ca}
\address[Hyde]{Mathematics and Statistics, University of Victoria, Victoria V8W 2Y2, Canada.}

\thanks{This research has been partially supported by the Carl Zeiss Foundation and by DFG grant PE 2299/3-1 (Alberto Espuny Díaz) and by the UK Research and Innovation Future Leaders Fellowship MR/S016325/1 and ERC Advanced Grant 101020255 (Joseph Hyde).}

\date{\today}

\begin{document}

\begin{abstract}
Let $G$ be a graph obtained as the union of some $n$-vertex graph $H_n$ with minimum degree $\delta(H_n)\geq\alpha n$ and a $d$-dimensional random geometric graph $G^d(n,r)$.
We investigate under which conditions for $r$ the graph $G$ will a.a.s.\ contain the $k$-th power of a Hamilton cycle, for any choice of $H_n$.
We provide asymptotically optimal conditions for $r$ for all values of $\alpha$, $d$ and $k$.
This has applications in the containment of other spanning structures, such as $F$-factors.

\end{abstract}
\maketitle

\section{Introduction}

\subsection{Randomly perturbed graphs}

Initiated by \citet{BFM03}, the study of randomly perturbed graphs involves taking a deterministic graph $H$ (usually with some minimum degree condition) and a random graph~$G$ (on the same vertex set) that both fail to have some property $P$, and seeing if the union of these graphs in fact has property $P$. 
In \cite{BFM03}, the property of interest is containing a Hamilton cycle, that is, a cycle covering all vertices of the graph.
A seminal result of \citet{Dirac52} asserts that any $n$-vertex graph ($n\geq3$) with minimum degree at least $n/2$ contains a Hamilton cycle, while in $G(n,p)$,\footnote{For $n\in\mathbb{N}$ and $p\in[0,1]$, $G(n,p)$ is the random graph obtained by considering a vertex set of size $n$ and adding each of the $\binom{n}{2}$ possible edges with probability $p$, independently of each other.} \citet{posa76}, and independently \citet{Kor77}, famously proved the asymptotic threshold for Hamiltonicity to be $\log n/n$, that is, there exists a constant $C > 0$ such that, if $p \geq C\log n/n$, then $G(n,p)$ asymptotically almost surely (a.a.s.)\footnote{By `asymptotically almost surely' we mean `with probability tending to $1$ as $n$ tends to infinity'.} contains a Hamilton cycle.
(Both these results are tight; in particular, for some constant $c > 0$, a.a.s.\ $G(n,p)$ is not even connected for $p \leq c\log n/n$.) 
Interpolating between these results, \citet{BFM03} proved that, for any $\alpha \in (0, 1/2)$, there exists a constant $C = C(\alpha)$ such that, for any $n$-vertex graph $H_n$ with minimum degree $\delta(H_n)\geq\alpha n$, the union $H_n \cup G(n,C/n)$ a.a.s.\ contains a Hamilton cycle (and this is also best possible, up to the value of $C$).
Observe how the threshold for Hamiltonicity in $G(n,p)$ is improved upon by a logarithmic factor in this result. 

In recent years, the study of randomly perturbed (hyper)graphs has seen significant interest.
Hamiltonicity has been considered in randomly perturbed directed graphs~\cite{BFM03,KKS16}, hypergraphs~\cite{HZ20, KKS16, MM18} and subgraphs of the hypercube~\cite{CEGKO20}. 
An increasing number of other properties and situations involving spanning structures have also been studied, including containing powers of Ha\-mil\-ton cycles~\cite{ADRRS21,ADR22,BMPP20, DRRS20, BPSS22}, $F$-factors\footnote{An $F$-factor is a collection of vertex-disjoint copies of $F$ covering all vertices of a graph.}~\cite{BTW19, BPSS20, HMT19+, BPSS21}, spanning trees~\cite{JK19, KKS17, BHKMPP19} and general bounded degree spanning graphs~\cite{BMPP20};
coloured randomly perturbed graphs containing rainbow spanning structures~\cite{AF19, AH21, BFP21, AHL21};
finding minimum weight structures in weighted randomly perturbed graphs~\cite{frieze21},
and playing maker-breaker~\cite{CHMP21makerbreaker} and waiter-client~\cite{CHMP21waiterclient} games on randomly perturbed graphs. 

All of these results deal with perturbing a (hyper/di)graph by some binomial random structure, such as $G(n,p)$, or its Erd\H{o}s-Rényi $G(n,m)$ counterpart.
Very recently, \citet{EG21} studied Hamiltonicity in graphs perturbed by a random regular graph, and \citet{Espuny21} also studied Hamiltonicity in graphs perturbed by a random geometric graph.
In this paper, it is the latter setting we are interested in. 

\subsection{Random geometric graphs}

We define (labelled) random $\ell_p$-geometric graphs in $d$ dimensions with radius $r$ as follows.
Let $V \coloneqq \{1, \ldots, n\}$ and let $X_1, \ldots, X_n$ be $n$ independent uniform random variables on $[0,1]^d$. 
Let $E\coloneqq\{\{i,j\}:\lVert X_i - X_j\rVert_p \leq r\}$. 
We denote the resulting graph $(V,E)$ by $G^d_p(n,r)$.
For simplicity, we focus on the Euclidean metric and write $G^d(n,r)$ for $G_2^d(n,r)$; our results, however, will extend to all $1\leq p\leq\infty$ (see \Cref{rem:lp}).

Let us briefly consider the history of Hamilton cycles in random geometric graphs.
The sharp threshold for connectivity in $G^2(n,r)$ was determined to be $r^* = (\log n/(\pi n))^{1/2}$ by \citet{Penrose97} (that is, for all $\epsilon > 0$, if $r \geq (1 + \epsilon)r^*$, then a.a.s.\ $G^2(n,r)$ is connected, and if $r < (1 - \epsilon)r^*$, then a.a.s.\ $G^2(n,r)$ is not connected).
Díaz, Mitsche and Pérez-Giménez~\cite{DMP07} showed the sharp threshold for $G^2(n,r)$ containing a Hamiltonian cycle is the same as the sharp threshold for connectivity.
This result was greatly strengthened independently by \citet{BBKMW11} and \citet{MPW11}, who showed that, for $d\geq 2$, a.a.s.\ $G^d(n,r)$ contains a Hamiltonian cycle precisely when it becomes $2$-connected, implying a sharp threshold of $(c\log n/n)^{1/d}$ for some constant $c = c(d)$.
See the works of \citet{Pen03, Pen16} for more information on random geometric graphs.

In this paper, we consider graphs perturbed by a random geometric graph.
Let us introduce the necessary definitions.
For a graph $H_n$ on $n$ vertices, we label $V(H_n)$ with $\{1, \ldots, n\}$ and set $V(H_n \cup G^d(n,r)) \coloneqq \{1, \ldots, n\}$.
Now, let $X_1, \ldots, X_n$ be $n$ independent uniform random variables on $[0,1]^d$ and set $E(H_n \cup G^d(n,r))\coloneqq E(H_n)\cup \{\{i,j\}:\lVert X_i - X_j\rVert \leq r\}$.
We refer to $H_n\cup G^d(n,r)$ as $H_n$ \emph{perturbed} by a $d$-dimensional random geometric graph.
The threshold for a property in this setting of randomly perturbed graphs is defined as follows.
For $\alpha\in[0,1]$, let $\mathcal{H}_\alpha^n$ denote the set of all $n$-vertex graphs $H_n$ with $\delta(H_n)\geq\alpha n$.
Let $P$ be a monotone property for graphs.
We say that a function $r^*_{\alpha,d}(n)$ is a $G^d(n,r)$-\emph{perturbed threshold for $P$} if
\begin{itemize}
    \item for every $H_n\in\mathcal{H}_\alpha^n$, $\lim_{n\to\infty}\mathbb{P}[H_n\cup G^d(n,r)\in P]=1$ whenever $r=\omega(r^*_{\alpha,d})$, and
    \item there exists some $H_n\in\mathcal{H}_\alpha^n$ such that $\lim_{n\to\infty}\mathbb{P}[H_n\cup G^d(n,r)\in P]=0$ whenever $r=o(r^*_{\alpha,d})$.
\end{itemize}
(The formal definition requires the use of sequences of graphs on an increasing number of vertices, on which the limit is taken; for simplicity, we avoid this formulation.
We also note that the definition above is analogous to the definition of perturbed thresholds in graphs perturbed by $G(n,p)$; to differentiate between both, we will refer to these as $G(n,p)$-perturbed thresholds.)

\citet{Espuny21} recently provided the threshold for Hamiltonicity in graphs perturbed by a random geometric graph.

\begin{theorem}\label{thm:hamcycle}
    For every integer\/ $d \geq 1$ and\/ $\alpha \in (0, 1/2)$,\/ $n^{-1/d}$ is a\/ $G^d(n,r)$-perturbed threshold for Hamiltonicity.
\end{theorem}

\Cref{thm:hamcycle} follows from Theorem~1 and remarks in Section~4 in \cite{Espuny21}, and can be seen as an analogue of the original result of \citet{BFM03} for graphs perturbed by $G(n,p)$.

\subsection{Powers of Hamilton cycles}

For any integer $k\geq1$, the $k$-th power of a graph $G$ is obtained by adding an edge between any two vertices whose distance in $G$ is at most $k$.
Powers of Hamilton cycles have received much attention in the literature.

\citet{KSS96, KSS98} proved that, for $n$ sufficiently large, any $n$-vertex graph with minimum degree at least $kn/(k+1)$ contains the $k$-th power of a Hamilton cycle.
The threshold for the containment of the $k$-th power ($k\geq2$) of a Hamilton cycle in $G(n,p)$ is $n^{-1/k}$: this was determined by \citet{KO12} for $k\geq3$ (applying an earlier result of \citet{Ri00}) and, very recently, by \citet{KNP21} for $k=2$.
\citet{DRRS20} proved that, if $\delta(H_n)\geq kn/(k+1)$, then $H_n\cup G(n, C/n)$ a.a.s.\ contains a $(k+1)$-th power of a Hamilton cycle,
before \citet{NT21} strengthened this result, forcing the $(2k+1)$-th power of a Hamilton cycle in the same situation. 
\citet{ADRRS21}, and more recently \citet{ADR22}, expanded in several ways on the result in~\cite{DRRS20}.
Very recently, \citet{BPSS22} completely solved the problem for $H_n\cup G(n, p)$ containing the square of a Hamilton cycle, giving the $G(n,p)$-perturbed threshold when $H_n$ has minimum degree $\alpha n$ for $\alpha \in (0, 1/2)$.
Interestingly, this perturbed threshold exhibits a `jumping' behaviour in terms of $\alpha$, presenting an infinite number of `jumps' (such jumps had previously been observed in the $G(n,p)$-perturbed thresholds for the containment of clique~\cite{HMT19+} and cycle~\cite{BPSS21} factors in $H_n\cup G(n, p)$, but only finitely many).

In contrast with the behaviour in binomial random graphs, where the threshold is different for each $k\geq1$, the threshold for the containment of the $k$-th power of a Hamilton cycle in random geometric graphs is $(\log n/n)^{1/d}$ independently of the value of $k$.
Indeed, simply observe that, if $G^d(n,r)$ contains a Hamilton cycle $\mathcal{C}$, then, by the triangle inequality, $G^d(n,kr)$ must contain the $k$-th power of $\mathcal{C}$.
This simple argument cannot be used to extend \Cref{thm:hamcycle}, however, as any Hamilton cycle in this setting will use edges of $H_n$, which may have arbitrary lengths.

We extend \Cref{thm:hamcycle} in this paper, proving the following.

\begin{theorem}\label{thm:threshold}
For any integers\/ $d,k\geq1$ and\/ $\alpha\in(0,k/(k+1))$,\/ $n^{-1/d}$ is a\/ $G^d(n,r)$-perturbed threshold for the containment of the\/ $k$-th power of a Hamilton cycle.
\end{theorem}

Observe that the $G^d(n,r)$-perturbed threshold in \Cref{thm:threshold} is the same as that in \Cref{thm:hamcycle}.
Indeed, this accords with the behaviour of $G^d(n,r)$.
Note also that this improves upon the threshold in the purely random setting by a logarithmic factor.
The behaviour of the threshold is very different when compared with the same problem in graphs perturbed by $G(n,p)$: indeed, the $G^d(n,r)$-perturbed threshold is the same for all values of $k$, and there are no jumps in the behaviour of the threshold as a function of $\alpha$.

In order to prove \Cref{thm:threshold}, we will establish the following.

\begin{theorem}\label{thm:kpower}
For any integers\/ $d,k\geq1$ and\/ $\alpha\in(0,1)$, there exists a constant\/ $C$ such that the following holds.
Let\/ $H_n$ be an\/ $n$-vertex graph with minimum degree at least\/ $\alpha n$ and let\/ $r\geq(C/n)^{1/d}$.
Then, a.a.s.\ $H_n \cup G^d(n,r)$ contains the\/ $k$-th power of a Hamilton cycle.
\end{theorem}

This provides the upper bound for the threshold. 
We will provide a construction for the lower bound in \Cref{sec:f-factor}, after proving \Cref{thm:kpower}.

Our proof of \Cref{thm:kpower} builds on ideas which have been successful in proving results about Hamiltonicity in random geometric graphs~\cite{DMP07,BBKMW11,MPW11,Espuny21,BBPP2017,FP20}.
We partition the hypercube $[0,1]^d$ into smaller cubes, called \emph{cells}, of side comparable with $r$, and analyse the number of vertices that fall into each of these cells.
The side is chosen so that most cells will contain ``many'' vertices, which provides high local connectivity.
On the other hand, a global structure using the cells is used to obtain global connectivity.
In our case, the random geometric graph we are working with is below the threshold for connectivity, so the desired global structure cannot be obtained purely in $G^d(n,r)$; to achieve the desired global connectivity we will use some of the edges of $H_n$, adapting an absorption idea introduced in \cite{Espuny21}.

\subsection{Applications}

Our main result has a series of corollaries that extend several lines of research in randomly perturbed graphs into graphs perturbed by a random geometric graph.

The first of these lines deals with $F$-factors, for some fixed graph $F$.
See the works of \citet{BTW19}, \citet{HMT19+}, and \citet{BPSS20,BPSS21} for some results in this setting.
Observe that, for any fixed graph $F$, there must exist a constant $k$ such that any $|V(F)|$ consecutive vertices of a $k$-th power of a Hamilton cycle contain a copy of $F$. 
Therefore, the following is a direct consequence of \Cref{thm:kpower}.

\begin{corollary}\label{coro1}
For any integer\/ $d\geq1$, any\/ $\alpha\in(0,1)$, and any fixed non-empty graph\/ $F$, there exists a constant\/ $C$ such that the following holds.
Let\/ $H_n$ be an\/ $n$-vertex graph with minimum degree at least\/ $\alpha n$, where\/ $|V(F)|$ divides\/ $n$, and let\/ $r\geq(C/n)^{1/d}$.
Then, a.a.s.\ $H_n \cup G^d(n,r)$ contains an\/ $F$-factor.
\end{corollary}

In \Cref{sec:f-factor}, we establish that $n^{-1/d}$ is indeed a $G^d(n,r)$-perturbed threshold for containing an $F$-factor for all $\alpha \in (0, 1-1/\chi_{\mathrm{cr}}(F))$ (see \Cref{sec:f-factor} for the definition of $\chi_{\mathrm{cr}}(F)$).
This improves by a logarithmic factor the threshold for the containment of an $F$-factor in purely $G^d(n,r)$, for many choices of $F$ (in particular, for all $F$ which have no isolated vertices).

A different direction is that of universality.
We say that an $n$-vertex graph is $2$-universal if it contains every graph on $n$ vertices with maximum degree at most $2$, that is, every graph which is a union of cycles and paths.
The problem of determining when graphs perturbed by $G(n,p)$ are $2$-universal has been considered by \citet{Parc20} and fully resolved by \citet{BPSS22}.
In our setting, since the square of a Hamilton cycle is $2$-universal,  \Cref{thm:kpower} yields the following.

\begin{corollary}\label{coro2}
For any integer\/ $d\geq1$ and\/ $\alpha\in(0,1)$, there exists a constant\/ $C$ such that the following holds.
Let\/ $H_n$ be an\/ $n$-vertex graph with minimum degree at least\/ $\alpha n$ and let\/ $r\geq(C/n)^{1/d}$.
Then, a.a.s.\ $H_n \cup G^d(n,r)$ is\/ $2$-universal.
\end{corollary}

Again, $n^{-1/d}$ is actually a $G^{d}(n,r)$-perturbed threshold for $2$-universality for all $\alpha\in(0,2/3)$ (every graph $H_n$ with $\delta(H_n) \geq2n/3$ is $2$-universal, as proved by \citet{AB93}).
The lower bound follows from the fact that this is a lower bound for the $G^{d}(n,r)$-perturbed threshold for the containment of a triangle-factor (see \Cref{sec:f-factor}).
As above, this improves on the threshold in the purely random setting by a logarithmic factor.

In more generality, we may consider the concept of bandwidth.
For $k \in \mathbb{N}$, we say a graph $G$ on $n$ vertices has bandwidth $k$ if we can label the vertices of $G$ with $\{1, \ldots, n\}$ such that for each edge $\{i,j\} \in E(G)$ we have $|i-j|\leq k$. 
Since the $k$-th power of a Hamilton cycle contains every graph of bandwidth at most $k$, \Cref{thm:kpower} immediately implies the following result.

\begin{corollary}\label{coro3}
For any integers\/ $d,k\geq1$ and\/ $\alpha\in(0,1)$, there exists a constant\/ $C$ such that the following holds.
Let\/ $H_n$ be an\/ $n$-vertex graph with minimum degree at least\/ $\alpha n$, let\/ $F$ be an\/ $n$-vertex graph of bandwidth at most\/ $k$, and let\/ $r\geq(C/n)^{1/d}$.
Then, a.a.s.\ $H_n \cup G^d(n,r)$ contains a copy of\/ $F$.
\end{corollary}

We remark that we do not try to optimise the constant $C$ which is obtained in \Cref{thm:kpower}, and it is unlikely that it is best possible.
By extension, the same is true of \Cref{coro1,,coro2,,coro3}.
It would be interesting to investigate whether $H_n \cup G^d(n,r)$ can have one of the properties previously mentioned, but not another.
For example, does $H_n \cup G^d(n,r)$ contain the square of a Hamilton cycle precisely when it becomes $2$-universal, or does $2$-universality appear substantially earlier?
To be more precise, for any $\alpha>0$ and $d\in\mathbb{N}$, do there exist positive constants $C_1<C_2$ such that, for any $n$-vertex graph $H_n$ with $\delta(H_n)\geq\alpha n$, a.a.s.\ $H_n\cup G^d(n,(C_1/n)^{1/d})$ is $2$-universal but, for some $n$-vertex graph $H_n$, a.a.s.\ $H_n\cup G^d(n,(C_2/n)^{1/d})$ does not contain the square of a Hamilton cycle?
A similar question is relevant for any pair of properties which exhibit the same $G^d(n,r)$-perturbed threshold.

\section{Preliminaries}\label{section2}

\subsection{Notation}

For any $n \in \mathbb{N}\cup\{0\}$, we denote $[n]\coloneqq\{i\in\mathbb{N}:1\leq i\leq n\}$ (in particular, $[0]=\varnothing$).
Given any set $S$ and any $k \in \mathbb{N}$, we write $S^{(k)}$ to denote the set of all subsets of $S$ of size $k$.
For parameters $a$ and $b$, whenever we claim that a statement holds for $0<a\ll b\leq 1$, called a \emph{hierarchy}, we mean that there exists an (unspecified) non-decreasing function $f\colon\mathbb{R}\to\mathbb{R}$ such that the claim holds for all $0<b\leq 1$ and for all $0<a\leq f(b)$.
This generalises naturally to longer hierarchies, and also to hierarchies where one parameter may depend on two or more other parameters.
A sequence of events $\{\mathcal{E}_n\}_{n\in\mathbb{N}}$ is said to hold \emph{asymptotically almost surely} (a.a.s.~for short) if $\mathbb{P}[\mathcal{E}_n]\to1$ as $n\to\infty$.
In all asymptotic statements, we will ignore rounding issues whenever these do not affect the arguments.

Most of our graph theoretic notation is standard.
The vertex set and edge set of a graph $G$ are denoted by $V(G)$ and $E(G)$, respectively.
We always consider labelled graphs, meaning that whenever we say that $G$ is an $n$-vertex graph we implicitly assume that $V(G)=[n]$.
If $G$ is a geometric graph (meaning here that each of its vertices is assigned to a position in $\mathbb{R}^d$ for some integer $d$), then $V(G)$ may interchangeably be used to refer to the set of positions which the vertices of $G$ are assigned to, and similarly the notation $v$ may refer to a vertex or its position.
We usually abbreviate the notation for edges $e=\{u,v\}$ as $e=uv$.
Given any vertex $v\in V(G)$, we define its \emph{neighbourhood} $N_G(v)\coloneqq\{u\in V(G):uv\in E(G)\}$ and its \emph{degree} $d_G(v)\coloneqq|N_G(v)|$.
We denote the minimum and maximum vertex degrees of $G$ by $\delta(G)$ and $\Delta(G)$, respectively.
Given a set $S\subseteq V(G)$, we write $N^{\cap}_G(S)\coloneqq\bigcap_{v\in S}N_G(v)$ for the \emph{common neighbourhood} of $S$.
Given a graph $G$ and two disjoint sets of vertices $A,B\subseteq V(G)$, we denote by $G[A]$ the graph on vertex set $A$ whose edges are all edges of $G$ which have both endpoints in $A$, and by $G[A,B]$ the graph on vertex set $A\cup B$ whose edges are all edges of $G$ which have one endpoint in $A$ and the other in $B$.
A \emph{path} $P$ is a graph whose vertices can be labelled in such a way that $E(P)=\{v_iv_{i+1}:i\in[|V(P)|-1]\}$.
If the endpoints of a path (the first and last vertices in the labelling described above) are~$u$ and~$v$, we sometimes refer to it as a \emph{$(u,v)$-path}.
Given a $(u,v)$-path $P$ and a $(v,w)$-path $P'$ such that $V(P)\cap V(P')=\{v\}$, we write~$PP'$ to denote the path obtained by concatenating $P$ and $P'$ (formally, this is the union graph of~$P$ and~$P'$).
If $P'$ is a single edge $vw$, we will write this as $Pw$.
Multiple concatenations will be written in the same way. 

\subsection{Azuma's inequality}

Let $\Omega$ be a set (we will later take $\Omega=[0,1]^d$), and let $f\colon \Omega^n\to\mathbb{R}$ be some function.
For some positive $L\in\mathbb{R}$, we say that $f$ is \emph{$L$-Lipschitz} if, for all $x,y\in\Omega^n$ such that $x$ and $y$ are identical in all but one coordinate, we have that $|f(x)-f(y)|\leq L$.
In order to bound the deviations of certain random variables, we consider the following consequence of Azuma's inequality (see, e.g.,~\cite[Corollary~2.27]{JLR}).

\begin{lemma}\label{lem:Azuma}
Let\/ $X_1,\ldots,X_n$ be independent random variables taking values in a set\/ $\Omega$.
Let\/ $f\colon\Omega^n\to\mathbb{R}$ be an\/ $L$-Lipschitz function.
Then, for any\/ $t\geq0$, the random variable\/ $X\coloneqq f(X_1,\ldots,X_n)$ satisfies that
\[\mathbb{P}[X\geq\mathbb{E}[X]+t]\leq \exp\left(-\frac{t^2}{2L^2n}\right)\qquad\text{ and }\qquad\mathbb{P}[X\leq\mathbb{E}[X]-t]\leq \exp\left(-\frac{t^2}{2L^2n}\right).\]
\end{lemma}

\section{Proof of \texorpdfstring{\Cref{thm:kpower}}{Theorem 1.2}}\label{section3}

We begin with the following simple lemma, which will allow us to find large complete bipartite graphs later on.

\begin{lemma}\label{lemma:LargeCommonNeigh}
For every integer\/ $k\geq1$ and\/ $\alpha\in(0,1)$, there exist some\/ $\eta>0$ and positive integers\/ $L,n_0$ such that the following holds.
Let\/ $H_n$ be a graph on\/ $n\geq n_0$ vertices with\/ $\delta(H_n)\geq\alpha n$.
Let\/ $S\subseteq V(H_n)$ be any set of vertices with\/ $|S|\geq L$.
Then, there exists some\/ $T\subseteq S$ with\/ $|T|=k$ such that\/ $|N^{\cap}_{H_n}(T)|\geq\eta n$.
\end{lemma}

\begin{proof}
Let $L\geq4k/\alpha$, $n_0\geq2L/\alpha$ and $\eta\leq\alpha/(4\inbinom{L}{k})$.
Let $S\subseteq V(H_n)$ be any subset with $|S|=L$.

First, we restrict ourselves to the bipartite graph $H'\coloneqq H_n[S,V(H_n)\setminus S]$.
By our choice of parameters, for all $v\in S$ we have $d_{H'}(v)\geq\alpha n/2$\COMMENT{We have $d_{H'}(v)\geq\alpha n-L\geq\alpha n-\alpha n/2=\alpha n/2$.}, while for all $v\in V(H_n)\setminus S$ we have $d_{H'}(v)\leq L$.
Now let $S'\coloneqq\{v\in V(H_n)\setminus S:d_{H'}(v)\geq k\}$.
It follows by double counting the edges of $H'$ that $|S'|\geq\alpha n/4$.\COMMENT{We have that
\[\alpha Ln/2\leq|E(H')|\leq|S'|L+(n-|S'|)k\implies|S'|(L-k)\geq(\alpha L/2-k)n\implies|S'|\geq\frac{(\alpha L/2-k)n}{L-k}\geq\frac{1}{4}\alpha n,\]
where for the last inequality it suffices to have $L\geq4k/\alpha$.}

Now consider an auxiliary bipartite graph $H''$ with parts $S'$ and $S^{(k)}$ where, for any $v\in S'$ and $T\in S^{(k)}$, $vT\in E(H'')$ whenever $T\subseteq N_{H'}(v)$ (that is, whenever $v\in N^{\cap}_{H'}(T)$).
It follows by the definition of $S'$ that $d_{H''}(v)\geq1$ for every $v\in S'$, so by the bound on $|S'|$ we have that $|E(H'')|\geq\alpha n/4$.
But this means that the average degree of the vertices of $S^{(k)}$ in $H''$ is at least $\alpha n/(4\inbinom{L}{k})\geq\eta n$.
In particular, there must be some $T\in S^{(k)}$ with $d_{H''}(T)\geq\eta n$, as we wanted to prove.
\end{proof}

\begin{proof}[Proof of \Cref{thm:kpower}]
Let $0<1/n_0\ll1/C\ll\eta\ll1/L\ll1/d,1/k,\alpha\leq1$, where $\eta$ and $L\in\mathbb{N}$ are given by \Cref{lemma:LargeCommonNeigh} with input $2k$ and $\alpha$.\footnote{Note that our $n_0$ here may be significantly larger in value than the $n_0$ found in \Cref{lemma:LargeCommonNeigh}.}
Let $R\coloneqq L+k\cdot3^d$.
Throughout, we may assume that $n\geq n_0$, as we are aiming to show $H_n\cup G^d(n,r)$ a.a.s.~contains the $k$-th power of a Hamilton cycle.

Let $r\coloneqq(C/n)^{1/d}$, $s\coloneqq1/\lceil2\sqrt{d}/r\rceil$ and $K\coloneqq s^dn$.
Partition $[0,1]^d$ into \mbox{$d$-dimensional} hypercubes of side~$s$\COMMENT{So there are $s^{-d}$ cells.} (intuitively, $s$ is close to $r/(2\sqrt{d})$, and our choice of~$C$ ensures that~$K$ is sufficiently large for all the ensuing claims to hold; in particular, we have $1/K\ll\eta,1/R$).
We refer to each of the smaller $d$-dimensional hypercubes as a \emph{cell}, and denote the set of all cells by $\mathcal{C}$.
We say that two distinct cells $c_1,c_2\in\mathcal{C}$ are \emph{friends} if their boundaries intersect (in particular, if the boundaries share a single point, they count as intersecting).
It follows that each cell is friends with at most $3^d-1$ other cells.
Given any set of points $S\subseteq[0,1]^d$, we say that a cell is \emph{dense} in~$S$ if it contains at least~$R$ points from $S$, and we call it \emph{sparse} in $S$ otherwise.

Consider a labelling of the vertices of $V(H_n)$ as $v_1,\ldots,v_n$, and let $X_1,\ldots,X_n$ be independent uniform random variables on $[0,1]^d$.
Consider the random geometric graph $G=G^d(n,r)$ on the vertex set of $H_n$ obtained by assigning position $X_i$ to $v_i$.
The event that any of the $X_i$ lie in the boundary of any of the cells has measure $0$, so we may assume that each vertex lies in the interior of some cell.
Note that, by the definition of~$s$, if some $v\in V(H_n)$ lies in a cell $c\in\mathcal{C}$, then it is joined by an edge of $G$ to all other vertices in~$c$, as well as to all vertices in cells which are friends of~$c$.\footnote{Indeed, this follows from our choice of $s$ and the fact that the diagonal of a cell has length $\sqrt{d}s$.}\COMMENT{Given a fixed cell, its friends are those which are incident to it (i.e., their boundaries intersect in some way, corners are enough).
Thus, it suffices to have that two diagonals of cells are at most $r$ (the worst case would be given when we want to see if a vertex in a corner of a cell is connected to a vertex in the furthest corner of a friend which is as far as possible, that is, it only shares a corner with a cell that shares a corner with the first cell).
The diagonal of a cell has length $\sqrt{d}s$, so we want to enforce that $2\sqrt{d}s\leq r$, and this follows by the choice of $s$.}
For any vertex $v\in V(H_n)$, we say that a cell $c\in\mathcal{C}$ is \emph{$v$-dense} (with respect to $H_n$) if it contains at least $2k$ of the neighbours of $v$ in~$H_n$.
Otherwise, we say that it is \emph{$v$-sparse}.
We provide a similar definition for larger sets of vertices: given any $S\subseteq V(H_n)$, we say that a cell $c\in\mathcal{C}$ is \emph{$S$-dense} (with respect to $H_n$) if there is a subset $T\subseteq S$ of size $2k$ such that $|c\cap N^{\cap}_{H_n}(T)|\geq2k$, and we call it \emph{$S$-sparse} otherwise.\footnote{For clarity, note that a $v$-dense/$v$-sparse (resp.\ $S$-dense/$S$-sparse) cell need not contain $v$ (resp.\ $S$).}

\begin{claim}\label{claim:fewsparse}
The following properties hold a.a.s.:
\begin{enumerate}[label=$(\mathrm{\roman*})$]
    \item\label{claim:fewsparse1} The number of cells which are sparse in\/ $V(H_n)$ is at most\/ $\nume^{-K/2}n$.
    \item\label{claim:fewsparse2} For each vertex\/ $v\in V(H_n)$, the number of\/ $v$-sparse cells is at most\/ $\nume^{-\alpha K/2}n$.
    \item\label{claim:fewsparse3} For every set\/ $S\subseteq V(H_n)$ of size\/ $|S|=L$, the number of\/ $S$-sparse cells is at most\/ $\nume^{-\eta K/2}n$.
\end{enumerate}
\end{claim}

\begin{claimproof}
The proofs of the three statements proceed in the exact same way, with a few more details needed for \ref{claim:fewsparse3}.
We therefore omit the details for \ref{claim:fewsparse1} and \ref{claim:fewsparse2}, and only discuss the details for \ref{claim:fewsparse3}.

Given any positive integer $m$, let $g_m\colon([0,1]^d)^m\to\mathbb{Z}_{\geq0}$ be a function which, given a set of $m$ points $x_1,\ldots,x_m\in[0,1]^d$, returns the number of cells which contain at most $2k-1$ of these points.
Clearly, $g_m$ is $1$-Lipschitz for every $m\in\mathbb{N}$.

Fix a set $S\subseteq V(H_n)$ with $|S|=L$.
By \Cref{lemma:LargeCommonNeigh}, there exists some $T\subseteq S$ with $|T|=2k$ such that $|N^{\cap}_{H_n}(T)|\geq\eta n$.
Therefore, for a cell to be $S$-dense, it suffices that it contains $2k$ of the vertices in $N^{\cap}_{H_n}(T)$.
Let $m\coloneqq|N^{\cap}_{H_n}(T)|$, and let $N^{\cap}_{H_n}(T)=\{v_{i_1},\ldots,v_{i_m}\}$.
For each $u\in V(H_n)$ and each cell $c\in\mathcal{C}$, we have that $\mathbb{P}[u\in c]=s^d$; thus, since the variables $X_i$ are independent, for any cell $c$ we have that\COMMENT{We have that 
\[\mathbb{P}[c\text{ is }S\text{-sparse}]\leq\sum_{i=0}^{2k-1}\binom{m}{i}s^{di}(1-s^d)^{m-i}.\]
Now we substitute $s^d=K/n$ and use the bounds $\eta n\leq m\leq n$ and $1-s^d\leq\nume^{-s^d}$.
With this, together with the fact that $i<2k$ and we may choose $n$ arbitrarily large (in particular, $\eta n-2k\geq2\eta n/3$), we may bound 
\[\binom{m}{i}\leq\binom{n}{i}\leq\frac{n^i}{i!}\qquad\text{ and }\qquad (1-s^d)^{m-i}\leq\nume^{-s^d(\eta n-i)}\leq\nume^{-2K\eta/3}.\]
Altogether, this yields
\[\mathbb{P}[c\text{ is }S\text{-sparse}]\leq\sum_{i=0}^{2k-1}\frac{n^i}{i!}\frac{K^i}{n^i}\nume^{-2\eta K/3}=\sum_{i=0}^{2k-1}\frac{K^i}{i!}\nume^{-2\eta K/3}.\]
Now, since we may assume $K>2k$, the expression above is increasing in $i$, so we may further bound
\[\mathbb{P}[c\text{ is }S\text{-sparse}]\leq2k\frac{K^{2k-1}}{(2k-1)!}\nume^{-2\eta K/3}\leq K^{2k}\nume^{-2\eta K/3}\leq \nume^{-\eta K/2}.\]
Observe that the second-to-last bound holds trivially for $K\geq 2k$, and the last one holds whenever $K/\log K\geq12k/\eta$.}
\[\mathbb{P}[c\text{ is }S\text{-sparse}]\leq\sum_{i=0}^{2k-1}\binom{m}{i}s^{di}(1-s^d)^{m-i}\leq\sum_{i=0}^{2k-1}\frac{K^i}{i!}\nume^{-2\eta K/3}\leq\nume^{-\eta K/2}.\]

Now let $Z\coloneqq g_{m}(X_{i_1},\ldots,X_{i_m})$ be the number of $S$-sparse cells, so $\mathbb{E}[Z]\leq \nume^{-\eta K/2}n/K$ (recall the number of cells is $s^{-d} = n/K$).
Since $g_m$ is $1$-Lipschitz, by \Cref{lem:Azuma} we conclude that\COMMENT{We have that
\begin{align*}
    \mathbb{P}\left[Z\geq\nume^{-\eta K/2}n\right]&=\mathbb{P}\left[Z\geq \nume^{-\eta K/2}n/K+(K-1)\nume^{-\eta K/2}n/K\right]\\
    &\leq\mathbb{P}\left[Z\geq\mathbb{E}[Z]+(K-1)\nume^{-\eta K/2}n/K\right]\leq\exp\left(-(K-1)^2\nume^{-\eta K}n/2K^2\right)=\nume^{-\Theta(n)}.
\end{align*}
}
\[\mathbb{P}[Z\geq\nume^{-\eta K/2}n]\leq \nume^{-\Theta(n)}.\]
The statement follows by a union bound over all sets $S\subseteq V(H_n)$ with $|S|=L$.
\end{claimproof}

Condition on the event that $G$ satisfies the properties of the statement of \Cref{claim:fewsparse}, which holds a.a.s.
Let $\mathcal{C}_{\mathrm{s}}$ be the set of cells which are sparse in $V(H_n)$, and let $\mathcal{C}_{\mathrm{d}}\coloneqq\mathcal{C}\setminus\mathcal{C}_{\mathrm{s}}$.
We define an auxiliary graph $\Gamma$ with vertex set $\mathcal{C}_{\mathrm{d}}$ where two cells are joined by an edge whenever they are friends.
In particular, $\Delta(\Gamma)\leq 3^d-1$.

\begin{claim}\label{claim:fewcomponents}
The number of connected components of\/ $\Gamma$ is at most\/ $\nume^{-K/3}n$.
\end{claim}

\begin{claimproof}
Say that a component of $\Gamma$ is \emph{large} if it contains at least $1/(6s)=n^{1/d}/(6K^{1/d})$ cells, and that it is \emph{small} otherwise.
Clearly, the number of large components of $\Gamma$ is at most $6(n/K)^{(d-1)/d}\leq\nume^{-K/3}n/2$, since $1/n\ll 1/K$.

Now consider any small component of $\Gamma$.
We claim that the number of sparse cells which are friends with some cell of this component is at least $2^d-1$.
This lower bound is achieved when the component consists of a single cell and this cell lies in a ``corner'' of the hypercube.

In order to prove the bound in full generality, choose an arbitrary cell $c_0$ of a small component of $\Gamma$ and choose a corner of the hypercube which is at distance at least $1/3$ from $c_0$ in each direction.
Let us assume without loss of generality that said corner is $\{1\}^d$.
Now, we follow an iterative process to define a sequence of cells.
For each $i\geq0$, if the hypercube $C_i$ of side length $2s$ having $c_i$ in a corner and growing in all directions towards $\{1\}^d$ contains some dense cell other than $c_i$, we choose one such cell $c_{i+1}\subseteq C_i$; otherwise, the process ends\COMMENT{The process would also end if $C_i$ is not contained in $[0,1]^d$, but this will not occur, as we discuss below.}.
Note that, whenever the process does not stop, $c_{i+1}$ is a translate of $c_i$ by a vector all whose coordinates are non-negative (at least one of them being positive) and at most $s$.
In particular, $c_{i+1}$ is a friend of $c_i$ and, thus, in the same component of $\Gamma$.
The process results in a sequence of distinct dense cells $c_0,c_1,\ldots,c_{t}$, where, for each $i\in[t]$, $c_i$ is a translate of $c_{0}$ by a vector all whose coordinates are non-negative and at most $is$.
Moreover, $c_0,c_1,\ldots,c_{t}$ all lie in the same component of $\Gamma$.
Now, we must have $t<1/(6s)$, as otherwise the component containing $c_0$ would be large.
But this means that $c_t$ is a translate of $c_0$ by a vector each of whose components is at most $1/6$, which in turn means that $c_t$ is at distance at least $1/6$ from each of the coordinate hyperplanes through $\{1\}^d$.
Thus, the process must have stopped because the hypercube $C_t$ of side length $2s$ having $c_t$ in a corner and growing in all directions towards $\{1\}^d$ contains only sparse cells (other than $c_t$).
There are exactly $2^d-1$ such cells, and all of them are friends with $c_t$, which proves our desired bound.

On the other hand, trivially no sparse cell can be friends with more than $3^d-1$ cells which lie in distinct components of $\Gamma$.
Further, by \Cref{claim:fewsparse}\ref{claim:fewsparse1}, there are at most $\nume^{-K/2}n$ sparse cells. Hence, a double counting argument\COMMENT{Consider a bipartite graph with vertex sets the sparse cells and the components of $\Gamma$, where two are joined by an edge if they are friends of each other.
By double counting the number of edges, it follows that the number of small components is at most
\[\frac{3^d-1}{2^d-1}\nume^{-K/2}n\leq2^d\nume^{-K/2}n<\nume^{-K/3}n/2,\]
where the first inequality holds since $d\geq1$ and the last inequality holds since $1/K\ll1/d$ (in fact, it suffices to have $K\geq6d+6$, since $2^{d+1}\nume^{-K/2}n<\nume^{d+1-K/2}n\leq\nume^{-K/3}n$ whenever $d+1\leq K/6$).} guarantees that the number of small components is at most 
\[\frac{3^d-1}{2^d-1}\nume^{-K/2}n\leq 2^d\nume^{-K/2}n\leq \nume^{-K/3}n/2,\]
and the proof is complete.
\end{claimproof}

We are now going to construct the $k$-th power of a Hamilton cycle in $H_n\cup G$.
Roughly speaking, for each connected component of $\Gamma$, we may find the $k$-th power of a cycle, as a subgraph of $G$, spanning all vertices which lie in the cells of this component.
Then, by using some edges of $H_n$, we will incorporate all leftover vertices into said powers of cycles, before combining the different structures into a single spanning $k$-th power of a cycle.
To make this process easier, however, it is better to proceed in a different order.
First (see step~\hyperlink{algostep1}{1} below), we choose the edges of $H_n$ which we will use to incorporate all vertices which lie in sparse cells.
Second (see step~\hyperlink{algostep2}{2}), we choose the edges of $H_n$ which will be used to combine the powers of cycles that we will find in each component of $\Gamma$.
Then (see step~\hyperlink{algostep3}{3}), we actually construct the $k$-th power of a spanning cycle in each component. Whilst constructing these $k$-th powers of spanning cycles we make sure that the endpoints of the edges of $H_n$ which we set aside earlier appear in an order that facilitates the incorporation of vertices in sparse cells and the grafting together of powers of cycles in different components.
Hence, we will be able to use these edges of $H_n$ to combine the different structures into a single spanning $k$-th power of a cycle.

We begin by setting up some notation.
Let $\boldsymbol{s}\coloneqq\bigcup_{c\in\mathcal{C}_{\mathrm{s}}}c\subseteq[0,1]^d$.
Given any cell $c\in\mathcal{C}_{\mathrm{d}}$, let $\Gamma(c)$ be the connected component of $\Gamma$ which contains $c$.
Further, we initialise sets $\mathcal{F}\coloneqq\mathcal{C}_{\mathrm{s}}$ and $\mathcal{F}^*\coloneqq\{c\in\mathcal{C}_{\mathrm{d}}:|V(\Gamma(c))|=1\}$.
By \Cref{claim:fewsparse}\ref{claim:fewsparse1} and \Cref{claim:fewcomponents} we have that $|\mathcal{F}\cup\mathcal{F}^*|\leq2\nume^{-K/3}n$.
Both $\mathcal{F}$ and $\mathcal{F}^*$ constitute sets of ``forbidden'' cells which we will avoid at certain times when connecting vertices from different cells via edges of~$H_n$.
These sets will be updated as we choose edges of $H_n$ to construct our spanning structure.
Indeed, each time we choose a set of edges of $H_n$ joining two distinct cells, we ensure afterwards that both cells belong to $\mathcal{F}$.
Roughly speaking, this will guarantee that the edges needed to join two given cells do not interact with the edges needed to join another pair of cells.
Furthermore, these edges are always chosen to join to a dense cell (hence why $\mathcal{C}_s\subseteq\mathcal{F}$).
In what follows, $\mathcal{F}^*$ will only ever contain dense cells. 
In particular, if in what follows a component has only one cell not in $\mathcal{F}$, then we will add this cell to $\mathcal{F}^*$.
Further, we will always have
\begin{equation}\label{equa:Fbound}
    |\mathcal{F}\cup\mathcal{F}^*|\leq \nume^{-K/6}n.
\end{equation}
Indeed, this bound will follow from the initial bound and the fact that we update $\mathcal{F}$ and $\mathcal{F}^*$ at most $2R\nume^{-K/3}n$ times, and each time the size of their union will increase by at most $3$.\COMMENT{$8R\nume^{-K/3}n\leq \nume^{-K/6}n$ since $1/K \ll 1/R$  (it suffices to have $K\geq6\log(8R)$).}

We first define some ``absorbing sets'' which will be used to incorporate all vertices in sparse cells into a $k$-th power of a Hamilton cycle.
We define these iteratively.
For each $v\in V(H_n)\cap\boldsymbol{s}$ (of which, by \Cref{claim:fewsparse}\ref{claim:fewsparse1}, there are at most $R|\mathcal{C}_{\mathrm{s}}|\leq R\nume^{-K/2}n$), we proceed as follows.
\begin{enumerate}[label=\arabic*.]
    \item \hypertarget{algostep1}{Choose} an arbitrary $v$-dense cell $c'(v)\in\mathcal{C}\setminus(\mathcal{F}\cup\mathcal{F}^*)$; note that such a cell exists by \Cref{claim:fewsparse}\ref{claim:fewsparse2} and \eqref{equa:Fbound}.\COMMENT{By \eqref{equa:Fbound}, we have that $|\mathcal{C}\setminus(\mathcal{F}\cup\mathcal{F}^*)|\geq s^{-d}-\nume^{-K/6}n\geq 3s^{-d}/4$ (note, furthermore, that we only choose dense cells because $\mathcal{C}_{\mathrm{s}}\subseteq\mathcal{F}$).
    By \Cref{claim:fewsparse}\ref{claim:fewsparse2}, the number of $v$-dense cells is at least $s^{-d}-\nume^{-\alpha K/2}n\geq 3s^{-d}/4$.
    Therefore, there must be some cell which belongs to both sets.
    This same argument will be used repeatedly, implicitly.}
    Choose any set $T_v\subseteq N_{H_n}(v)\cap c'(v)$ of size $2k$,\COMMENT{Such a set exists by definition, since $c'(c)$ is $v$-dense.} and split it into two equal-sized sets $T_v^1$ and $T_v^2$.
    Note that $G[T_v]$ is a complete graph, and so $(H_n\cup G)[\{v\}\cup T_v]$ is also complete.
    Then, add $c'(v)$ to $\mathcal{F}$.\COMMENT{The reason for this set is as follows. When choosing edges of $H_n$, we always choose sets of edges such that their endpoints form a clique in the cell $c'(v)$. Now, if $c'(v)$ contained endpoints of several sets of edges of $H_n$, then it could be that the edges inside the cell do not form a ``linear clique-forest''. In such a case, our proof would not work. To avoid this, we simply make sure that no cell is used for more than one set of edges of $H_n$ joining it to other cells.}
    Moreover, if $|V(\Gamma(c'(v)))\setminus\mathcal{F}|=1$, add this remaining cell to $\mathcal{F}^*$.\COMMENT{This is done to ensure that we connect everything.
    Indeed, if we do not avoid this cell in the future, we might get stuck with a component which only contains e.g. two dense cells.
    This is the main reason to work with $\mathcal{F}^*$.}
\end{enumerate}
Once this process is finished, let $F_{\mathrm{s}}\coloneqq\bigcup_{v\in V(H_n)\cap\boldsymbol{s}}T_v$ and $\mathcal{A}_{\mathrm{s}}\coloneqq\mathcal{F}\setminus\mathcal{C}_{\mathrm{s}}$.
Note that $\mathcal{A}_{\mathrm{s}}$ corresponds precisely to the set of cells which contain a set $T_v$, for some $v\in V(H_n)\cap\boldsymbol{s}$.

Consider an auxiliary graph $\Gamma^*$ which we initiate as $\Gamma$.
We are next going to modify this graph $\Gamma^*$ into a connected graph.
We will simultaneously construct an auxiliary tree $\mathcal{T}$ whose vertex set is the set of components of $\Gamma$, which we now initialise as an empty graph.
We will update $\Gamma^*$ in $t-1$ steps, where $t\leq \nume^{-K/3}n$ is the number of components of~$\Gamma$ (see \Cref{claim:fewcomponents}).
In each of these steps, we will add exactly one edge to $\Gamma^*$, connecting two of its components, and exactly one edge to $\mathcal{T}$ joining the same components.
The auxiliary edges of $\Gamma^*$ will correspond to where we will later connect the $k$-th powers of cycles which we will construct in each component; we build the structure necessary for this at the same time as we update $\Gamma^*$.
Our initial choice of $\mathcal{F}$ and $\mathcal{F}^*$, together with how we have and will update them, will be crucial in guaranteeing that the upcoming process can be carried out.
In particular, throughout this process we will satisfy the following three properties:
\begin{enumerate}[label=(P\arabic*)]
    \item\label{prop1} $\mathcal{F}$ and $\mathcal{F}^*$ are disjoint;
    \item\label{prop2} a component $\gamma$ of $\Gamma^*$ is contained in $\mathcal{F}\cup\mathcal{F}^*$ if and only if $|V(\gamma)\setminus\mathcal{F}|=1$, and
    \item\label{prop3} if a component of $\Gamma^*$ intersects $\mathcal{F}^*$, then it is fully contained in $\mathcal{F}\cup\mathcal{F}^*$.
\end{enumerate}
Observe that these three properties hold at the end of step~\hyperlink{algostep1}{1} by construction.
Given any cell $c\in\mathcal{C}_{\mathrm{d}}$, let $\Gamma^*(c)$ denote the connected component of~$\Gamma^*$ which contains~$c$.
Initialise a set of vertices $F_{\mathrm{d}} \coloneqq \varnothing$.
For each $i\in[t-1]$, we proceed as follows.
\begin{enumerate}[label=\arabic*.,start=2]
\item \hypertarget{algostep2}{Choose} a smallest component $\gamma$ of $\Gamma^*$, and choose an arbitrary cell $c\in V(\gamma)\setminus\mathcal{F}$ (which exists since $V(\gamma)\nsubseteq\mathcal{F}$ by \ref{prop1} and \ref{prop2}).
Let $S\subseteq V(H_n)\cap c$ be an arbitrary set of size $L$.
Choose an arbitrary $S$-dense cell $c'=c'(c)\in\mathcal{C}\setminus(\mathcal{F}\cup\mathcal{F}^*\cup V(\gamma))$; its existence follows from \Cref{claim:fewsparse}\ref{claim:fewsparse3}, \eqref{equa:Fbound} and the fact that $\gamma$ is a smallest component of~$\Gamma^*$.\COMMENT{The fact that $\gamma$ is a smallest component of $\Gamma^*$ means that it has at most $s^{-d}/2$ cells.
Then, we are removing from consideration at most $s^{-d}/2+\nume^{-K/6}n+\nume^{-\eta K/2}n<3s^{-d}/4$ cells, so there are lots of cells to choose from.}
In particular, by \ref{prop2}, it follows that\COMMENT{$\Gamma^*(c)$ and $\Gamma^*(c')$ are vertex disjoint. The first one is missing at least one vertex, and the second one, since it avoids $\mathcal{F}^*$, must be missing at least two vertices.}
\begin{equation}\label{eq:newprop}
    |(V(\Gamma^*(c))\cup V(\Gamma^*(c')))\setminus\mathcal{F}|\geq3,
\end{equation}
and by \ref{prop3}, that
\begin{equation}\label{eq:newprop2}
    V(\Gamma^*(c'))\cap\mathcal{F^*}=\varnothing.
\end{equation}

By the definition of an $S$-dense cell, there exist sets $T_c\subseteq V(H_n)\cap c$ and $T_{c'}\subseteq V(H_n)\cap c'$ with $|T_c|=|T_{c'}|=2k$ such that $H_n[T_c,T_{c'}]$ is a complete bipartite graph.
Note that it follows that $(H_n\cup G)[T_c\cup T_{c'}]$ is a complete graph.
Split the sets $T_c$ and $T_{c'}$ into equal sized parts $T_c=T_c^1\cup T_c^2$ and $T_{c'}=T_{c'}^1\cup T_{c'}^2$.
Add all vertices in $T_c$ and $T_{c'}$ to $F_{\mathrm{d}}$.
Add the edge $\{c,c'\}$ to $\Gamma^*$, so now these two cells are in the same component, and add $\{\Gamma(c),\Gamma(c')\}$ to $\mathcal{T}$.

If $c\in\mathcal{F}^*$, remove it from this set.
Then, add $c$ and $c'$ to $\mathcal{F}$.
Observe that $\mathcal{F}$ and $\mathcal{F}^*$ remain disjoint, so \ref{prop1} holds.
Moreover, note that the current component $\Gamma^*(c) = \Gamma^*(c')$ cannot be fully contained in $\mathcal{F}$ by \eqref{eq:newprop}.
If $|V(\Gamma^*(c))\setminus\mathcal{F}|=1$, add this remaining cell to $\mathcal{F}^*$.
Together with \eqref{eq:newprop2} and \ref{prop2} and \ref{prop3} applied to the original $\gamma$, this guarantees that \ref{prop2} and \ref{prop3} hold.
\end{enumerate}
Each iteration of step~\hyperlink{algostep2}{2} reduces the number of components of $\Gamma^*$ and $\mathcal{T}$ by one, so it follows that, after we perform all iterations, $\Gamma^*$ is connected and $\mathcal{T}$ is a tree (also, $\mathcal{F}^*$ is empty by \ref{prop3}).
Once this is achieved, let $\mathcal{A}_{\mathrm{d}}$ correspond precisely to the set of cells which contain one of the sets $T_c$ or $T_{c'}$ obtained during the iterations of step~\hyperlink{algostep2}{2}, and observe that $F_{\mathrm{d}}$ contains all vertices in the sets $T_c$ and $T_{c'}$.

Let $F\coloneqq F_{\mathrm{s}}\cup F_{\mathrm{d}}$ and note that, by construction, for each $c\in\mathcal{C}$ we have that $|F\cap c|\in\{0,2k\}$, with $|F\cap c|=2k$ if and only if $c\in\mathcal{A}_{\mathrm{s}}\cup\mathcal{A}_{\mathrm{d}}=\mathcal{F}$.
We now give a (cyclic) labelling to the vertices of $H_n$ lying in the cells of each component $\gamma$ of $\Gamma$ in such a way that any two vertices whose labels differ by at most $k$ will be joined by an edge of $G$ (that is, this order provides a $k$-th power of a cycle spanning all vertices in $\gamma$).
We represent said labelling by a directed cycle in $G$.
We make sure that, if there is any cell $c\in V(\gamma)\cap(\mathcal{A}_{\mathrm{s}}\cup\mathcal{A}_{\mathrm{d}})$, then the vertices in $c$ are ordered in such a way to facilitate either the incorporation of a vertex in some sparse cell or the grafting together of $k$-th powers of cycles in two different components, dependent on whether $c \in \mathcal{A}_{\mathrm{s}}$ or $c \in \mathcal{A}_{\mathrm{d}}$. 
For each component $\gamma$ of $\Gamma$, we proceed as follows.
\begin{enumerate}[label=\arabic*.,start=3]
\item \hypertarget{algostep3}{Let} $T\subseteq\Gamma$ be a spanning tree of $\gamma$.
In particular, $\Delta(T)<3^d$.
Fix an arbitrary cell $c_0\in V(\gamma)$.
Consider an arbitrary traversal of~$T$ which, starting at $c_0$, goes through every edge of $T$ twice and ends in the starting cell (for instance, this may be achieved by performing a depth-first search on $T$ with $c_0$ as a root).
This traversal takes $m\coloneqq2(|V(\gamma)|-1)$ steps, each step corresponding to an edge of $T$.
We use this traversal to construct a directed cycle $\mathfrak{C}(\gamma)$ as follows.

Choose a vertex $x_0\in (V(H_n)\cap c_0)\setminus F$ and let $P_0\coloneqq x_0$; this will be the beginning of a path which we will grow into $\mathfrak{C}(\gamma)$.
For notational purposes, set $V(P_{-1})\coloneqq\varnothing$.
We now inductively create paths $P_1,\ldots,P_m$, each of which is obtained by extending the previous path, and which satisfy that their endpoint $x_i\notin F$ and, for the cell $c$ containing $x_i$, $(V(P_i)\setminus V(P_{i-1}))\cap c=\{x_i\}$.
Note that these two properties also hold for the base case $i=0$.

For each $i\in[m]$ we extend $P_{i-1}$ to a path $P_i$ as follows.
Let~$c$ be the current cell in our traversal of $T$, and let $x_{i-1}\in (V(H_n)\cap c)\setminus F$ be the last vertex of $P_{i-1}$ (that is, $V(P_{i-1})\setminus V(P_{i-2})$ intersects $c$ precisely in $x_{i-1}$).
Let $c'$ be the next cell of the traversal.
Because $c$ and $c'$ are friends, every vertex in $c'$ is joined to every vertex in $c$ by an edge of~$G$.
Choose an arbitrary vertex $x_i\in (V(H_n)\cap c')\setminus(F\cup V(P_{i-1}))$.
Now consider the following cases.
\begin{enumerate}[label=(\roman*)]
    \item If this is the last time that~$c$ is visited in the traversal of~$T$ and $c\notin\mathcal{A}_{\mathrm{s}}\cup\mathcal{A}_{\mathrm{d}}$, let $P_i'$ be any path with vertex set $(V(H_n)\cap c)\setminus V(P_{i-2})$ having $x_{i-1}$ as a starting point, and let $P_i\coloneqq P_{i-1}P_i'x_i$.
    \item\label{case2} If this is the last time that~$c$ is visited in the traversal of~$T$ and $c\in\mathcal{A}_{\mathrm{s}}$, then there is some $v\in V(H_n)\cap\boldsymbol{s}$ such that $T_v\subseteq V(H_n)\cap c$.
    If so, let $P_i'$ be any path with vertex set $(V(H_n)\cap c)\setminus V(P_{i-2})$ having $x_{i-1}$ as a starting point and such that $T_v$, $T_v^1$ and $T_v^2$ span subpaths of $P_i'$, with $T_v^1$ coming before $T_v^2$ in $P_i'$.
    Then, let $P_i\coloneqq P_{i-1}P_i'x_i$.
    \item\label{case3} If this is the last time that~$c$ is visited in the traversal of~$T$ and $c\in\mathcal{A}_{\mathrm{d}}$, then in step~\hyperlink{algostep2}{2} we defined a set $T_c$.
    Let $P_i'$ be any path with vertex set $(V(H_n)\cap c)\setminus V(P_{i-2})$ having $x_{i-1}$ as a starting point and such that $T_c$, $T_c^1$ and $T_c^2$ span subpaths of $P_i'$, with $T_c^1$ coming before $T_c^2$ in $P_i'$.
    Then, let $P_i\coloneqq P_{i-1}P_i'x_i$.
    \item Otherwise, choose an arbitrary set $U_i\subseteq(V(H_n)\cap c)\setminus(F\cup V(P_{i-1}))$ of size $k-1$, let $P_i'$ be an arbitrary path on vertex set $U_i\cupdot\{x_{i-1}\}$ with $x_{i-1}$ as a starting point, and let $P_i\coloneqq P_{i-1}P_i'x_i$.
\end{enumerate}

To complete the cycle, let $P'$ be an $(x_m,x_0)$-path whose internal vertices are all vertices of $(V(H_n)\cap c_0)\setminus V(P_{m})$.
Furthermore, similarly to cases~\ref{case2} and~\ref{case3} above, if $c_0\in\mathcal{A}_{\mathrm{s}}$ (resp.\ $c_0\in\mathcal{A}_{\mathrm{d}}$), we make it so that the corresponding sets $T_v$, $T_v^1$ and $T_v^2$ (resp.\ $T_{c_0}$, $T_{c_0}^1$ and $T_{c_0}^2$) span subpaths of $P'$ with $T_v^1$ (resp.\ $T_{c_0}^1$) coming before $T_v^2$ (resp.\ $T_{c_0}^2$) in $P'$.
We then set $\mathfrak{C}(\gamma)\coloneqq P_m\cup P'$, directing (and thus labelling) $\mathfrak{C}(\gamma)$ by following $P_m$ from $x_0$ to $x_m$ and then following $P'$ from $x_m$ to $x_0$.
(Note there is a degenerate case when $m=0$.
In this case, by the initial definition of $\mathcal{F}^*$ and the process we followed in step~\hyperlink{algostep2}{2}, we must have $c_0\in\mathcal{A}_{\mathrm{d}}$.
Then, $\mathfrak{C}(\gamma)$ is instead a cycle covering all vertices of $V(H_n)\cap c_0$ such that the corresponding sets $T_{c_0}$, $T_{c_0}^1$ and $T_{c_0}^2$ span subpaths of $\mathfrak{C}(\gamma)$, and $\mathfrak{C}(\gamma)$ is oriented so that $T_{c_0}^1$ comes before $T_{c_0}^2$.)

Observe that every cell of $\gamma$ contains at least~$R=L+k\cdot3^d$ vertices and is visited at most $3^d$ times throughout the traversal; this, together with the fact that $|F\cap c|\leq 2k < L$ for all $c\in\mathcal{C}$, guarantees that we can choose vertices as described throughout the process.
\end{enumerate}

We can finally combine all the properties obtained throughout the previous three steps in order to obtain the $k$-th power of a Hamilton cycle.

First, for each $v\in V(H_n)\cap\boldsymbol{s}$, insert $v$ between $T_v^1$ and $T_v^2$ in the labelling of the cycle $\mathfrak{C}(\gamma)$ containing $T_v$.
For simplicity, for each component $\gamma$ of $\Gamma$ we redefine $\mathfrak{C}(\gamma)$ to denote the cycles resulting from these insertions.
Note that, by the definition of the sets $T_v$ and the fact that they are disjoint, inserting these vertices preserves the property that there is an edge of $H_n\cup G$ joining each vertex in $\mathfrak{C}(\gamma)$ to the next $k$ vertices in this cycle.
Thus, at this point, for each component $\gamma$ of $\Gamma$, the cycle $\mathfrak{C}(\gamma)$ yields a $k$-th power of a cycle on its vertex set.
Furthermore, the union of $\mathfrak{C}(\gamma)$ over all components $\gamma$ now contains \emph{all} the vertices of $H_n$.
It only remains to graft them together into a (unique) spanning cycle.

To achieve this, we first perform a depth-first search on $\mathcal{T}$ (the auxiliary tree constructed in step~\hyperlink{algostep2}{2}), rooted at an arbitrary component $\gamma_1$, to give a labelling $\gamma_1,\ldots,\gamma_t$ to the components.
We are going to use induction to obtain a sequence of directed cycles $\mathfrak{C}_1,\ldots,\mathfrak{C}_t$ such that, for each $j\in[t]$, the cycle $\mathfrak{C}_j$ satisfies $V(\mathfrak{C}_j)=\bigcup_{\ell\in[j]}V(\mathfrak{C}(\gamma_\ell))$ and that each of its vertices is joined by an edge of $H_n\cup G$ to each of the $k$ following vertices in $\mathfrak{C}_j$.
The base case $j=1$ holds trivially by setting $\mathfrak{C}_1\coloneqq\mathfrak{C}(\gamma_1)$.

Now, for some $j<t$, assume that $\mathfrak{C}_j$ satisfies the desired properties and we want to construct $\mathfrak{C}_{j+1}$.
Since $\mathcal{T}$ is a tree and the labelling is given by a DFS on $\mathcal{T}$, there is a unique edge in $\mathcal{T}$ joining $\gamma_{j+1}$ to some $\gamma_\ell$ with $\ell\leq j$.
This edge corresponds precisely to one pair of cells $c,c'\in\mathcal{C}_{\mathrm{d}}$ which contain sets $T_c$, $T_{c'}$, as per step~\hyperlink{algostep2}{2}.
Assume, without loss of generality, that $c\in\gamma_{j+1}$.
Now, simply insert the vertices of $\mathfrak{C}(\gamma_{j+1})$, in the order specified by this cycle, starting at $T_c^2$ and ending at $T_c^1$, into $\mathfrak{C}_j$, between $T_{c'}^1$ and $T_{c'}^2$.
Clearly, the resulting cycle $\mathfrak{C}_{j+1}$ contains all the desired vertices.
Furthermore, the property that each of its vertices is joined by an edge of $H_n\cup G$ to each of the $k$ following vertices in $\mathfrak{C}_{j+1}$ follows from the fact that this is true of both $\mathfrak{C}_{j}$ and $\mathfrak{C}(\gamma_{j+1})$, the fact that $(H_n\cup G)[T_c\cup T_{c'}]$ forms a complete graph, and the fact that the pairs of sets $T_c$, $T_{c'}$ are pairwise-disjoint.

After this process is finished, $\mathfrak{C}_t$ is a spanning cycle which yields the desired $k$-th power of a Hamilton cycle.
\end{proof}

\begin{remark}\label{rem:lp}
We can extend \Cref{thm:kpower} to\/ $\ell_p$-random geometric graphs\/ $G_p^d(n,r)$ for all\/ $1\leq p\leq\infty$.
To achieve this, it suffices to adjust the definition of\/ $s$ at the beginning of the proof, to guarantee that any vertex in a cell\/ $c$ will be joined by an edge to all vertices in the same cell or in cells which are friends of\/ $c$.
In particular, for any\/ $1\leq p\leq\infty$, it would suffice to take\/ $s=1/\lceil2d/r\rceil$.\COMMENT{This value is necessary for $p=1$.}
\end{remark}

\section{Threshold lower bounds}\label{sec:f-factor} 

The following simple observation will come in useful for proving the lower bounds of the thresholds.

\begin{proposition}\label{prop:fewedges}
Let\/ $d\geq1$ be an integer.
If\/ $r=o(n^{-1/d})$, then a.a.s.\ $G^d(n,r)$ has\/ $o(n)$ edges.
\end{proposition}

\begin{proof}
For any two distinct indices $i,j\in[n]$, we have that $\mathbb{P}[\{i,j\}\in E(G^d(n,r))]=\Theta(r^d)$, where the implied constant depends only on $d$.
(Indeed, upon conditioning on each possible value of $X_i$, the probability lies between $\theta_dr^d/2^d$ and $\theta_dr^d$, where $\theta_d$ is the volume of the $d$-dimensional sphere of radius $1$. 
Undoing the conditioning guarantees that the desired probability lies between these two values as well.)
Therefore, $\mathbb{E}[|E(G^d(n,r))|]=\binom{n}{2}\cdot\Theta(r^d)=o(n)$.
The statement now follows by Markov's inequality.\COMMENT{Being very formal: whatever function $r$ is, we can always find another function $s$ with $s=\omega(r)$ but $s=o(n^{-1/d})$ (e.g., we could take the geometric mean of the two functions).
Now, by Markov's inequality,
\[\mathbb{P}[|E(G^d(n,r))|\geq n^2s^d]\leq\frac{\mathbb{E}[|E(G^d(n,r))|]}{n^2s^d}=\frac{\theta(n^2r^d)}{n^2s^d}=o(1),\]
and $n^2s^d=o(n)$, so we are done.}
\end{proof}

We begin by proving \Cref{thm:threshold} in full.

\begin{proof}[Proof of \Cref{thm:threshold}]
Since \Cref{thm:kpower} establishes the upper bound for the $G^d(n,r)$-perturbed threshold, it suffices to provide a lower bound.
We propose the following construction.
Fix $k\in\mathbb{N}$ and $\alpha\in(0,k/(k+1))$.
Let $H_n$ be an $n$-vertex complete $(k+1)$-partite graph with parts $V(H_n)=A_1\cupdot\ldots\cupdot A_k\cupdot B$, where $|B|=\max\{0,(\alpha-1)k+1\}n$ and $A_1,\ldots,A_k$ have the same size.
Observe that $\delta(H_n)\geq\alpha n$ by construction. 

Now, the $k$-th power of a Hamilton cycle contains a collection of vertex-disjoint copies of $K_{k+1}$ covering all but at most $k$ of its vertices. We will show that a.a.s.\ $H_n \cup G^d(n,r)$ with $r = o(n^{-1/d})$ does not contain such a collection, which establishes the lower bound we require.\COMMENT{We have two cases to look at.
If $\alpha<1-1/k$, then this holds trivially, as $B$ is empty, so $H_n$ is a balanced complete $k$-partite graph, and so $\delta(H_n)=(k-1)n/k>\alpha n$.
Otherwise, every vertex in $B$ has degree $n-|B|=(1-\alpha)kn>\alpha n$ (indeed, we are assuming that $\alpha<k/(k+1)$) and every other vertex $v\in A_i$ has degree \[d(v)=n-|A_i|=n-\frac{n-\beta n}{k}=\frac{(k-1)n+((\alpha-1)k+1)n}{k}=\alpha n.\]}
Now, define $\beta \coloneqq k/(k+1) - \max\{\alpha,(k-1)/k\} > 0$.
Then, rearranging, we have $|B| = n/(k+1) - \beta kn$. 
Thus, since each copy of $K_{k+1}$ in $H_n$ must contain precisely one vertex in $B$, we have that
\begin{enumerate}[label=$(\ast)$]
\item\label{item:ast} every collection of vertex-disjoint copies of $K_{k+1}$ in $H_n$ has size at most $n/(k+1) - \beta kn$.
\end{enumerate}

Assume that there exists a collection $\mathcal{K}$ of vertex-disjoint copies of $K_{k+1}$ covering all but at most $k$ vertices of $H_n \cup G^d(n,r)$ with $r = o(n^{-1/d})$.
Then, by \ref{item:ast}, at least $\beta kn - k$ copies of $K_{k+1}$ in $\mathcal{K}$ must contain at least one edge of $E(G^d(n,r))$, so $|E(G^d(n,r))| \geq \beta kn - k$.
But a.a.s.\ $|E(G^d(n,r))| = o(n)$ by Proposition~\ref{prop:fewedges}, and $\beta kn - k = \theta(n)$ as $k$ and $\beta = \beta(\alpha, k)>0$ are fixed constants.
Thus, a.a.s.\ $H_n \cup G^d(n,r)$ with $r = o(n^{-1/d})$ does not contain a collection of vertex-disjoint copies of $K_{k+1}$ covering all but at most $k$ of its vertices, hence it a.a.s.\ does not contain the $k$-th power of a Hamilton cycle.
\COMMENT{The hidden constant above is a function of $\alpha$ and $k$.
This is not best possible. 
Here is the intuitive idea that makes the constant depend only on $\alpha$.
Say that we describe the power of a Hamilton cycle by labelling the vertices in the order in which they must appear in said cycle.
The closest we can get to the $k$-th power of a Hamilton cycle is by changing from one part to a different one every time (as otherwise we already have a pair of consecutive vertices whose edge is missing, and this would already give us the desired $\Theta(n)$).
Say that we always go from one part to the next, following the order $A_1,\ldots,A_k,B$ (again, otherwise there are some edges that are needed for the $k$-th power that would be missing, so this is the best that we can do).
Since $B$ is smaller than the other parts (indeed, this is trivial if $\alpha<1-1/k$, and otherwise it holds since $\alpha<k/(k+1)$), however, at some point we will need to start skipping it when we follow the order.
From this point on, we will be missing an edge between any vertex and the $k$-th subsequent vertex, needed to obtain the desired power of a Hamilton cycle.}
\end{proof}

We remark that our graph $H_n$ is not the only possible extremal construction for this problem.
Indeed, \citet{DRRS20} used a different construction to obtain the lower bound for the $G(n,p)$-perturbed threshold for the containment of the $k$-th power of a Hamilton cycle when $\alpha\in((k-1)/k,k/(k+1))$.
Their construction would yield the same result as ours.

Let us now turn our attention to $F$-factors.
Let $r \in \mathbb{N}$.
\citet{HS70} proved that any graph $G$ on $n$ vertices with minimum degree $(1-1/r)n$, and $n$ divisible by $r$, contains a $K_r$-factor.
(The case $r = 3$ was previously established by \citet{CH63}, and the case $r=2$ follows from Dirac's theorem~\cite{Dirac52}.)
Observe that the chromatic number $\chi(K_r)$ of the clique satisfies $\chi(K_r) = r$. 
For any graph $F$, \citet{AY96} showed that any $n$-vertex graph $H_n$ with minimum degree $\delta(H_n)\geq(1-1/\chi(F))n$, and $n$ divisible by $|V(F)|$ sufficiently large, contains an $F$-factor. 
Although this minimum degree condition is best possible for cliques and some other graphs, it is not best possible for all graphs $F$. To get the full picture we need to introduce another parameter.

For a graph $F$ with $\chi(F) \geq 2$, we define the \emph{critical chromatic number} of $F$ to be
\[\chi_{\mathrm{cr}}(F) \coloneqq (\chi(F) - 1)\frac{|V(F)|}{|V(F)| - \sigma(F)},\]
where $\sigma(F)$ is the size of the smallest possible vertex class over all $\chi(F)$-colourings of $F$. 
We also define an \emph{$F$-tiling} to be a collection of vertex-disjoint copies of $F$. 
\citet{Komlos00} proved that, for any $\eta > 0$, there exists an integer $n_0 = n_0(\eta, F)$ such that, if $H_n$ is a graph on $n\geq n_0$ vertices and has minimum degree at least $(1 - 1/\chi_{\mathrm{cr}}(F))n$, then $H_n$ contains an $F$-tiling covering all but at most $\eta n$ vertices\footnote{Informally, an \emph{almost perfect $F$-tiling} or an \emph{almost $F$-factor}.}. 
As K\"uhn and Osthus showed~\cite{KO09}, for many graphs $F$ it is $\chi_{\mathrm{cr}}(F)$, rather than $\chi(F)$, which is the parameter governing the best possible minimum degree condition for forcing an $F$-factor.
They also introduced the parameter $\operatorname{hcf}(F)$ that governs whether $\chi(F)$ or $\chi_{\mathrm{cr}}(F)$ is the correct such choice.
The definition of $\operatorname{hcf}(F)$ is fairly technical and so we will not define it formally here (see \cite{KO09} for the full definition).
Essentially, graphs $F$ with $\operatorname{hcf}(F) = 1$ have the ability to overcome certain divisibility barriers, that is, it is `easier' to extend an almost perfect $F$-tiling to an $F$-factor for such graphs.

\begin{theorem}[\citet{KO09}]\label{thm:ko}
    For every graph $F$ there exists a constant $C = C(F) \geq -1$ such that the following holds.
    Let $H_n$ be a graph on sufficiently many vertices $n$ which has minimum degree
    \[\delta(H_n)\geq(1 - 1/\chi^*(F))n + C,\]
    where $\chi^*(F) = \chi_{\mathrm{cr}}(F)$ if\/ $\operatorname{hcf}(F) = 1$ and $\chi^*(F) = \chi(F)$ otherwise.
    Moreover, assume $n$ is divisible by $|V(F)|$.
    Then, $H_n$ contains an $F$-factor.
\end{theorem} 

Note that the additive constant in the minimum degree condition is necessary for some graphs $F$ (this can be seen by modifying examples in \cite{AF99, Komlos00}).

We now show that $n^{-1/d}$ is also a $G^d(n,r)$-perturbed threshold for containing an $F$-factor (for $\alpha \in (0, 1 - 1/\chi_{\mathrm{cr}}(F))$).

\begin{theorem}\label{thm:f-factorperturbed}
    Let $F$ be a fixed graph.
    For every integer $d \geq 1$ and $\alpha \in (0, 1 - 1/\chi_{\mathrm{cr}}(F))$, for $n$ divisible by $|V(F)|$, we have that $n^{-1/d}$ is a $G^d(n,r)$-perturbed threshold for containing an $F$-factor.
\end{theorem}

To prove \Cref{thm:f-factorperturbed}, we will use \Cref{prop:fewedges} together with an adaptation of an extremal example of \citet{Komlos00}.
This generalises the construction used to prove the lower bound for \Cref{thm:threshold}.

\begin{proof}[Proof of \Cref{thm:f-factorperturbed}]
Since \Cref{coro1} establishes the upper bound for the $G^d(n,r)$-perturbed threshold, it suffices to provide the lower bound.
Let $k \coloneqq \chi(F)$ and $\sigma \coloneqq \sigma(F)$, and define $\gamma\coloneqq\min\{1/(k-1),1-\alpha\}$ and $\beta\coloneqq1-(k-1)\gamma$.
The bounds on $\alpha$ in the statement of Theorem~\ref{thm:f-factorperturbed} guarantee that $\gamma \in ({1}/{\chi_{\mathrm{cr}}(F)},1/({k-1})]$ and $\beta \in [0, {\sigma}/{|V(F)|})$.\COMMENT{The bounds on $\gamma$ follow trivially. To check now the upper bound on $\beta$, we need to use the definition of the critical chromatic number: since $\gamma>1/\chi_{\mathrm{cr}}(F)$, we have
\[\beta=1-(k-1)\gamma=1-(\chi(F)-1)\gamma<1-\frac{\chi(F)-1}{\chi_{\mathrm{cr}}(F)}=1-\frac{|V(F)|-\sigma}{|V(F)|}=\frac{\sigma}{|V(F)|}.\]}
Note that $(\beta + (k-1)\gamma)n = n$.
Let $H_n$ be the complete $k$-partite graph with parts $V(H_n)=A_1\cupdot\ldots\cupdot A_{k-1}\cupdot B$, where $|B|=\beta n$ and $|A_1|=\ldots=|A_{k-1}|=\gamma n$.
Observe that $\delta(H_n) \geq \alpha n$. Since every copy of $F$ must contain at least $\sigma$ vertices in $B$, we have that 
\begin{enumerate}[label=$(\ast\ast)$]
\item\label{item:ast2} every collection of vertex-disjoint copies of $F$ in $H_n$ has size at most $\beta n/\sigma$.
\end{enumerate}

Suppose that $r = o(n^{-1/d})$.
Assume that $H_n \cup G^d(n,r)$ contains an $F$-factor $\mathcal{F}$.
Then, by \ref{item:ast2}, at least $(1/|V(F)| - \beta/\sigma)n$ copies of $F$ in $\mathcal{F}$ must contain at least one edge of $E(G^d(n,r))$.
Set $\delta \coloneqq 1/|V(F)| - \beta/\sigma$ and note that $\delta > 0$.
Thus $|E(G^d(n,r))| \geq \delta n$.
But a.a.s.\ $|E(G^d(n,r))| = o(n)$ by Proposition~\ref{prop:fewedges}, and $\delta n = \theta(n)$ since $\delta \coloneqq \delta(\alpha, F) > 0$ is a fixed constant.
Therefore, a.a.s.\ $H_n \cup G^d(n,r)$ with $r = o(n^{-1/d})$ does not contain an $F$-factor.
\end{proof}

Reflecting on how the minimum degree condition in \Cref{thm:ko} depends on either $\chi(F)$ or $\chi_{\mathrm{cr}}(F)$, we ask the following question.
\begin{question}
    Let $F$ be a graph with $\chi(F) \neq \chi_{\mathrm{cr}}(F)$ and $\operatorname{hcf}(F) \neq 1$.
    For $\alpha \in [1-1/\chi_{\mathrm{cr}}(F), 1-1/\chi(F))$, what is the $G^d(n,r)$-perturbed threshold for containing an $F$-factor?
\end{question}

\section*{Acknowledgement}

We are grateful to an anonymous referee for their helpful comments on this manuscript.

\bibliographystyle{mystyle} 
\bibliography{geopert}


\end{document}